\documentclass[11pt]{article}

\usepackage{pgf,tikz}
\usepackage{mathrsfs}
\usetikzlibrary{arrows}

\usepackage{todonotes}

\usepackage{graphicx} 
\graphicspath{{figures/}} 

\usepackage{fullpage, url,amsmath,amsfonts,amssymb,mathtools,mathrsfs,graphicx,  algorithm, float, sansmath,epstopdf,color,caption,enumitem,tabularx}
\usepackage[final]{pdfpages}
\usepackage{amsthm}
\usetikzlibrary{automata,topaths}
\usetikzlibrary{decorations.pathreplacing,shapes.misc}
\usepackage{fancyhdr}

\def\beq{\begin{equation}}
\def\eeq{\end{equation}}
\def\baq{\begin{eqnarray}}
\def\eaq{\end{eqnarray}}
\def\baqn{\begin{eqnarray*}}
\def\eaqn{\end{eqnarray*}}

\newcommand{\ball}{\mathbb{B}}

\usepackage{multirow}
\usetikzlibrary{calc,arrows}
\theoremstyle{plain}
\newtheorem{definition}{Definition}
\newtheorem{remark}{Remark}

\newtheorem{example}{Example}
\newtheorem{theorem}{Theorem}
\newtheorem{lemma}[theorem]{Lemma}

\usepackage{blindtext}

\usepackage[colorlinks,linkcolor=blue,citecolor=red]{hyperref}

\usepackage{xcolor, framed}

\newcommand{\R}{{\mathbb R}}

\newcommand{\interior}{{\rm int}\kern 0.06em}

\def\<{\langle}
\def\>{\rangle}

\usepackage{ulem}

\usepackage{subcaption}

\usepackage{pict2e}

\if
{

\usepackage[pageref]{backref}
\renewcommand*{\backrefalt}[4]{%
\ifcase #1 %
(Not cited)%
\or
(Cited on p.~#2)%
\else
(Cited on pp.~#2)%
\fi
}

}
\fi

\begin{document}
\title{ {Sliding Mode Observers for Set-valued Lur'e  Systems with Uncertainties Beyond Observational Range}}
\author{Samir Adly\thanks{Laboratoire XLIM, Universit\'e de Limoges,
123 Avenue Albert Thomas,
87060 Limoges CEDEX, France\vskip 0mm
Email: \texttt{samir.adly@unilim.fr}} \qquad 
 Jun Huang\thanks{School of Mechanical and Electrical Engineering, Soochow University, China\vskip 0mm
 Email: \texttt{cauchyhot@163.com}}\qquad
 Ba Khiet Le \thanks{Optimization Research Group, Faculty of Mathematics and Statistics, Ton Duc Thang University, Ho Chi Minh City, Vietnam\vskip 0mm
 E-mail: \texttt{lebakhiet@tdtu.edu.vn}}
 }
\date{}
\maketitle

\begin{abstract}
{In this paper, we introduce a new  sliding mode observer for Lur'e set-valued dynamical systems, 
particularly addressing challenges posed by uncertainties not within the standard range of observation. 
{Traditionally, {most of }Luenberger-like observers and  sliding mode observer  have been designed only for  uncertainties in the range of observation.
Central to our approach is the treatment of the uncertainty term  which we decompose into two 
components: the first part in the observation subspace and the second part in its complemented subspace. We establish that when the second part converges to zero, 
 an exact sliding mode observer for the system can be obtained.} In scenarios where this convergence 
does not occur, our methodology allows for the estimation of errors between the actual state and the observer 
state.  This leads to a practical interval estimation technique, valuable in situations where part of the 
uncertainty lies outside the observable range. Finally, we  show that our observer is also a  $T$- observer  as well as {a strong $H^\infty$ observer}.
}

\end{abstract}

observation subspace\\


\section{Introduction}
{
Differential inclusion serves as an extension of ordinary differential equations, playing a pivotal role as a mathematical model within applied mathematics and control theory.
A set-valued Lur'e dynamical system, a form of differential inclusion, incorporates interconnection feedback, broadening its scope to encompass various models frequently used in the analysis of nonsmooth dynamical systems. Such models include evolution variational inequalities, projected dynamical systems, relay systems, and complementarity systems.
Set-valued Lur'e dynamical systems have been the subject of thorough investigations in recent decades (we refer, for example, to \cite{Acary, ahl, ahl2, bg, bg2, brogliato, bt, cs, L1, Lurie, abc0, abc} and the references therein).
An observer is a mathematical model used to reconstruct complete information about the state variables from partially measurable data. Observers design  for set-valued Lur'e systems represents a significant practical challenge, alongside other fundamental issues like well-posedness and stability analysis, which have received extensive attention in the literature. 
Numerous references addressing observer design for set-valued Lur'e systems are available, with a predominant focus on two methodological approaches: Luenberger-like techniques \cite{bh, Huang, Huang1, Huang3, abc0, abc} and sliding mode methodologies \cite{al,L3}.
} 

{The  Luenberger-like observers for set-valued Lur'e systems was initially explored by Brogliato-Heemels in \cite{bh}, which is an extension of  the work of Arcak-Kokotovic for the nonlinear case \cite{Arcak}.   Brogliato-Heemels considered the differential inclusion 
\beq \left\{
\begin{array}{lll}
\; \dot{x}=Ax +Bw +Gu,\\
&&\\
\;w \in -\mathcal{F}(Cx), \\
&&\\
\; y=Fx,
\end{array}\right. 
\eeq
and proposed the observer 
\beq \left\{
\begin{array}{lll}
\; \dot{\tilde{x}}=(A-LF)\tilde{x} +B\tilde{w} +Ly+Gu,\\
&&\\
\;\tilde{w} \in -\mathcal{F}\Big((C-KF)\tilde{x}+Ky\Big) \\
\end{array}\right. 
\eeq
where $A, B, C, G, F, K$ are matrices with appropriate dimensions, $\mathcal{F}$ is a maximal monotone set-valued operator and $y$ is the partial measurable output. 
Later, it was extended  by Huang et al in \cite{Huang} to incorporate uncertainty as follows 
\beq \left\{
\begin{array}{lll}
\; \dot{x}=Ax +Bw +f_1(x,u)+f_2(x,u)\theta,\\
&&\\
\;w \in -\mathcal{F}(Cx), \\
&&\\
\; y=Fx,
\end{array}\right. \label{syshu}
\eeq
where $\theta$ is an unknown constant. The proposed observer is
\beq \left\{
\begin{array}{lll}
\; \dot{\tilde{x}}=(A-LF)\tilde{x} +B\tilde{w} +f_1(\tilde{x},u)+f_2(\tilde{x},u)\tilde{\theta}+Ly,\\
&&\\
\;\tilde{w} \in -\mathcal{F}\Big((C-KF)\tilde{x}+Ky\Big) \\
&&\\
\dot{\tilde{\theta}}=h (\tilde{x},u)(y-F\tilde{x}),\\
\end{array}\right.  \label{obsyshu}
\eeq
where the range condition $f_2^T(x,u)P=h(x,u)F $ is satisfied for some symmetric positive definite matrix $P$.
The key idea in \cite{Huang} is to approximate the unknown by using an additional ODE, aiming to achieve observer convergence with an unknown rate.
{On the other hand}, B. K. Le  \cite{L3} proposed  a sliding mode observer  for a more general class of set-valued Lur'e systems,  which extends the nonlinear case considered by Xiang-Su-Chu \cite{Xiang}, see also \cite{Spurgeon} for a survey. 
Following this, Adly-Le applied the technique to system (\ref{syshu}) in \cite{al}, offering clear advantages over Luenberger-like methods (\ref{obsyshu}). These advantages include improved assumptions and achieving exponential convergence for the sliding mode observer without the need to solve additional ordinary differential equations. Moreover, the authors proposed a new, efficient smoothing approximation for the sliding mode technique to reduce the chattering effect using time guiding functions.
} 

{
Nonetheless, a significant limitation of existing results using Luenberger-like or sliding mode observers (we refer, for example, to \cite{al, Huang, L3}) is the necessity of the range condition for uncertainties, meaning that uncertainties must fall within the range of $P^{-1}F^T$. An alternative approach involves constructing $H^\infty$ observers. However, $H^\infty$ observers only offer estimations for the entire process under zero initial conditions, rather than providing estimations at specific time instances (see, for instance, \cite{Huang3}). This limitation prompts us to introduce a sliding mode observer to attain time-specific estimations for the Lur'e system in the presence of uncertainties that do not fall within the observation range, as follows:
\beq \left\{
\begin{array}{lll}
\; \dot{x}=Ax +B\omega +f(x,u)+\xi(t),\\
&&\\
\;\omega \in -\mathcal{F}(Cx), \\
&&\\
\; y=Fx,
\end{array}\right. \label{sysh}
\eeq
where $A\in \R^{n\times n}, B\in \R^{n\times m}, C\in \R^{m\times n}, F\in \R^{p\times n}$ are given matrix, $x\in \R^n$ is the state variable, $\mathcal{F}: \R^m\rightrightarrows  \R^m$ is a maximal monotone operator,  $u\in \R^r$ is the control input and $y\in \R^p$ is the measurable output. The nonlinear functions $f$ is Lipschitz continuous and   $\xi\in \R^l$ is unknown. The model in (\ref{sysh}) is quite general since in the right-hand side of the ODE, one has the linear  $Ax$, the nonlinear $f(x,u)$, the set-valued $B\mathcal{F}(Cx)$ and the general unknown $\xi$.
}
 {The idea is to decompose $\xi$ into 2 parts: $\xi=\xi_1+\xi_2$ where  $\xi_1=P^{-1}F^T\xi$ is the projection of $\xi$ onto the observation subspace ${\rm Im}(P^{-1}F^T)$ and $\xi_2$ is in the complement of the observation subspace. We show that if $\xi_2$ is bounded, we can obtain a useful estimation for the state of the original system. Interval estimation is  practical and unavoidable if  the uncertainty $\xi$ is general and not in the range of observation. We refer here some interval observers results for some specific systems \cite{Huang4,Mazenca}.  In addition, if $\xi_2$ tends to vanish when the time is large, we obtain the exact sliding mode observer, i.e., the observer state converges to the original state asymptotically.  Finally, we show that our proposed observer is  {a  $T$- observer,  a notion that we introduce to obtain a time instant estimation for the error. We also provide a mild additional condition such that our observer is  a strong $H^\infty$ observer, an extension  of $H^\infty$ observers which does not require the  zero initial condition. $T$- observers and strong $H^\infty$ observers allow us to obtain the total process and time instant estimations of the errors. 
In practice, there is a preference for the insights provided by the characteristics of strong $H^\infty$ observers over those of $H^\infty$ observers.
 }
 } 

{The paper is structured as follows: In Section 2, we revisit established definitions and essential results  for our subsequent analyses.
Section 3 presents our proposal for a sliding mode observer for the original system, which simultaneously functions as a $T$- observer as well as  a strong $H^\infty$ observer. In Section 4, we give several numerical examples to validate the theoretical findings.
The paper concludes in Section 5, where we offer our closing remarks and explore potential future directions.}
\section{{Notations} and mathematical background}
We denote the  scalar product and the corresponding norm of Euclidean spaces  by $\langle\cdot,\cdot\rangle$ and $\|\cdot\|$  respectively. 
A matrix $P\in \R^{n\times n}$  is called positive definite,  written  $P> 0$, if there exists $a>0$ such that 
$$
\langle Px,x \rangle \ge a\|x\|^2,\;\;\forall\;x\in {\R^n}.
$$
The Sign functions in $\R^n$ is defined by 
$$
{\rm Sign}(x)= \left\{
\begin{array}{l}
\frac{x}{\Vert x\Vert} \;\;\;\;{\rm if}\;\; \;\;x\neq 0\\ \\
\ball \;\;\;\;\;\;\; {\rm if}\;\;\;\;\;x=0,
\end{array}\right.
$$
where $\ball$ denotes the closed  unit ball in $\R^n$.

A set-valued mapping $\mathcal{F}: \R^m \rightrightarrows \R^m$ is called $\it{monotone}$ if for all $x,y\in \R^m$ and $x^*\in \mathcal{F}(x),y^*\in \mathcal{F}(y)$, one has 
$$\langle x^*-y^*,x-y\rangle \ge 0.$$ 
 Furthermore, $\mathcal{F}$  is called $\it{maximal \;monotone}$ if there is no monotone  operator $\mathcal{G}$ such that the graph of $\mathcal{F}$ is contained strictly in  the graph of $\mathcal{G}$ (see, e.g, \cite{Ac,br}).\\

\noindent  Let us recall a  general version of Gronwall's inequality \cite{Showalter}.
\begin{lemma}\label{gronwall}
Let $T>0$ be given and $a(\cdot),b(\cdot)\in L^1([0,T];\R)$ with $b(t)\ge 0$ for almost all $t\in [0,T].$ Let the absolutely continuous function $w: [0,T]\to \R_+$ satisfy:
\beq
(1-\alpha)w'(t)\le a(t)w(t)+b(t)w^\alpha(t),\;\; a. e. \;t\in [0,T],
\eeq
where $0\le \alpha<1$. Then for all $t\in [0,T]$:
\beq
w^{1-\alpha}(t)\le w^{1-\alpha}(0){\rm exp}\Big(\int_{0}^t a(\tau)d\tau\Big)+\int_{0}^t{\rm exp}\Big(\int_{s}^t a(\tau)d\tau\Big)b(s)ds.
\eeq
\end{lemma}

\section{Sliding mode observer and the convergence analysis}
In this section, we  propose a  sliding mode observer for the system (\ref{sysh}) under the following assumptions:\\

\noindent \textbf{Assumption 1: } The set-valued operator $\mathcal{F}: \R^m\rightrightarrows \R^m$ is a  maximal monotone operator. \\

\noindent \textbf{Assumption 2: } The   functions $f:\R^n\times \R^r\to \R^n,\;{(x,u)\mapsto f(x,u)}$ is $L_f$ - Lipschitz continuous w.r.t $x$, i.e.:
$$
\Vert f(x_1,u)-f(x_2,u)  \Vert \le L _f\Vert x_1-x_2\Vert.
$$
\noindent \textbf{Assumption 3: } There exist $\epsilon>0$, $P\in \R^{n\times n}>0$, $L\in \R^{n\times p}$ and  $K\in \R^{m\times p}$  such that 
\baq\label{ine}
&&P(A-LF)+(A-LF)^TP+2L_f \Vert P \Vert I+2\epsilon  I \le 0,\\\label{pbc}
&& B^TP=C-KF.
\eaq
\noindent \textbf{Assumption 4: } The unknown $\xi$ is continuous and can be decomposed as $\xi(t)=P^{-1}F^T\xi_1(t)+\xi_2(t)$,  
where   $P^{-1}F^T\xi_1(t)$ is  the projection of $\xi$ onto ${\rm Im}(P^{-1}F^T)$. The terms $\xi_1$ and $\xi_2$ are unknown, bounded by known continuous positive functions $\kappa_1(t)$ and $\kappa_2(t)$ respectively, i.e., for all $t\ge 0$, we have 
$$
\vert \xi_1(t) \vert \le \kappa_1(t), \;\;\vert \xi_2(t) \vert \le \kappa_2(t).
$$
 \begin{remark}\normalfont
i) Assumption 4 is quite general, since it does not require the uncertainty in the range of $P^{-1}F^T$, i.e., the range of observation.  \\
ii)  Suppose that  Assumption 4 holds, then we have $FP^{-1}\xi_2=0$, i.e., $\xi_2$ is in the kernel of $FP^{-1}$.
 \end{remark}
\noindent The proposed  sliding mode observer for (\ref{sysh})  is 
\beq \left\{
\begin{array}{lll}
\; \dot{\tilde{x}}=A\tilde{x} +B\tilde{\omega} -Le_y+f_1(\tilde{x},u)-   P^{-1}F^T( \kappa_1 {\rm Sign}(e_y)+\frac{\kappa_3 e_y}{\Vert e_y\Vert^2+\delta }),\\
&&\\
\;\tilde{\omega} \in -\mathcal{F}(C\tilde{x}-Ke_y), \\
&&\\
\; \tilde{y}=F\tilde{x},
\end{array}\right.
\label{obs}
\eeq
where
\beq
e:=\tilde{x}-{x},\;\;  e_y: =Fe,
\eeq
for some given small $\delta>0$.
 Since the nonlinear $f$ and  the unknown $\xi$ are continuous, we can obtain the existence and uniqueness of the original system (\ref{sysh}) and the approximate observer (\ref{obs}) under some additional mild conditions (see, e.g., \cite{al,bh,Huang}). In the following, we  show that if $\xi_2$ is bounded, we obtain a useful estimation for the error $t\mapsto e(t)$. On the other hand,  if $\xi_2$ tends to vanish when the time is large, we  have an exact observer indeed. 
\begin{theorem}\label{mainth} Suppose that Assumptions 1--4 hold. Let $V(t)=\langle Pe(t),e(t)\rangle.$ Then 
\beq\label{totales}
\sqrt{V(t)}\le \sqrt{V(0)}{\rm exp}(\frac{-\epsilon t}{\lambda_{max}})+\frac{\Vert P\Vert }{\sqrt{\lambda_{min}}}\int_{0}^t{\rm exp}(\frac{\epsilon (s-t)}{\lambda_{max}}){\Vert \kappa_2(t)\Vert}ds,
\eeq
{where $\lambda_{max}$ and $\lambda_{min}$ are the largest and smallest eigenvalues of $P$ respectively.}
In addition\\

\begin{enumerate}
\item[{\rm (a)}] If $\kappa_2(\cdot)$ is bounded by some $k>0$, then 
\beq\label{estia}
\Vert e(t) \Vert \le\frac{\lambda_{max}\Vert P\Vert  k}{ {\lambda_{min}}\epsilon}+{\rm exp}(\frac{-\epsilon t}{\lambda_{max}})(\sqrt{\frac{V(0)}{\lambda_{min}}}-\frac{\lambda_{max}\Vert P\Vert  k}{\epsilon {\lambda_{min}}}).
\eeq
{Additionally, we introduce the attractive set $\Omega := [0, \frac{k \Vert P\Vert }{\epsilon}]$. Consequently, the error norm $\Vert e(t)\Vert$ converges to any neighborhood of $\Omega$ in finite time and remains within that vicinity.}\\

\item[{\rm (b)}]  If   $\kappa_2(t)\to 0$ as $t\to\infty$, then  $\Vert e(t) \Vert\to 0$ as $t\to\infty$, i.e.,  (\ref{obs}) is indeed an exact observer of (\ref{sysh}).\\
 
\item[{\rm (c)}] If $\Vert \kappa_2(t)\Vert\le k e^{-at}$ {for  all $t\ge 0$}, then $e(t)$ converges exponentially to zero.\\

\item[{\rm (d)}] If $\kappa_2\in L^2(0,\infty )$, then 
\baqn
\sqrt{V(t)}&\le& \sqrt{V(0)}{\rm exp}(\frac{-\epsilon t}{\lambda_{max}})+\frac{\Vert P\Vert  }{\sqrt{\lambda_{min}}}\sqrt{\int_0^t \kappa^2_2(s)ds}\sqrt{\frac{\lambda_{max}}{2\epsilon}(1-{\rm exp}(\frac{-2\epsilon t}{\lambda_{max}}))}\\
&\le&\sqrt{V(0)}{\rm exp}(\frac{-\epsilon t}{\lambda_{max}})+\frac{\Vert P\Vert \Vert \kappa_2\Vert_{L^2(0,\infty)}}{\sqrt{\lambda_{min}}}\sqrt{\frac{\lambda_{max}}{2\epsilon}(1-{\rm exp}(\frac{-2\epsilon t}{\lambda_{max}}))}.
\eaqn
\end{enumerate}
\end{theorem}

\begin{proof}
From (\ref{sysh}) and (\ref{obs}), we have 
\baq\nonumber
\dot{e}&\in&(A-LF)e-B(\tilde{\omega}-{\omega})+f(\tilde{x},u)-f({x},u)\\
&-&  P^{-1}F^T( \kappa_1 {\rm Sign}(e_y)+\frac{\kappa_3 e_y}{\Vert e_y\Vert^2 +\delta})-(P^{-1}F^T\xi_1(t)+\xi_2(t)). \label{der}
\eaq
With $V(t)=\langle Pe(t),e(t)\rangle$, one has 
\baq\nonumber
\frac{1}{2}\frac{dV}{dt}&=&\langle P\dot{e}, e\rangle=\langle P(A-LF)e+PB(\tilde{\omega}-{\omega}), e\rangle\\\nonumber
&+&\langle P(f_1(\tilde{x},u)-f_1({x},u)), e \rangle -(\kappa_1(t)-\Vert\xi_1(t)\Vert){\Vert e_y\Vert} - \langle P\xi_2(t),e\rangle\\
&-&\frac{\kappa_3 \Vert e_y \Vert^2}{\Vert e_y\Vert^2+\delta}.
\label{dv}
\eaq
From (\ref{pbc}) and the monotonicity of $ \mathcal{F}$, we have 
\baq
\langle PB(\tilde{\omega}-{\omega}), e\rangle=\langle \tilde{\omega}-{\omega}), B^TP e\rangle
=\langle \tilde{\omega}-{\omega}, (C-KF) e\rangle\le 0. 
\label{mn}
\eaq
On the other hand, using Assumption 2, we obtain 
\baq\label{Lips}
\langle P(f_1(\tilde{x},u)-f_1({x},u)), e \rangle\le L_1 \Vert P \Vert \Vert e\Vert^2 \label{et1}.
\eaq
Thus from (\ref{ine}), (\ref{der})-(\ref{Lips}), we deduce that 
\beq\label{bounded}
\frac{1}{2}\frac{dV}{dt}\le-\epsilon \Vert e\Vert ^2  +\Vert P\Vert \Vert \kappa_2(t)\Vert \Vert e(t)\Vert\le -\frac{\epsilon}{\lambda_{max}}V+\frac{\Vert P\Vert \Vert \kappa_2(t)\Vert}{\sqrt{\lambda_{min}}}\sqrt{V}.
\eeq
Using Lemma 1 with $\alpha=1/2$, we have
$$
\sqrt{V(t)}\le \sqrt{V(0)}{\rm exp}(\frac{-\epsilon t}{\lambda_{max}})+\frac{\Vert P\Vert }{\sqrt{\lambda_{min}}}\int_{0}^t{\rm exp}(\frac{\epsilon (s-t)}{\lambda_{max}}){\Vert \kappa_2(s)\Vert}ds.
$$
(a) If $\kappa_2(\cdot)$ is bounded by some $k>0$, then 
\baqn
\sqrt{\lambda_{min}}\Vert e(t) \Vert \le\sqrt{V(t)}&\le& \sqrt{V(0)}{\rm exp}(\frac{-\epsilon t}{\lambda_{max}})+\frac{\Vert P\Vert  k}{ \sqrt{\lambda_{min}}}\int_0^t {\rm exp}(\frac{\epsilon (s-t)}{\lambda_{max}})ds.\\
&=&\sqrt{V(0)}{\rm exp}(\frac{-\epsilon t}{\lambda_{max}})+\frac{\lambda_{max}\Vert P\Vert  k}{\epsilon \sqrt{\lambda_{min}}}\big(1-{\rm exp}(\frac{-\epsilon t}{\lambda_{max}})\big)\\
&=&\frac{\lambda_{max}\Vert P\Vert  k}{\epsilon \sqrt{\lambda_{min}}}+{\rm exp}(\frac{-\epsilon t}{\lambda_{max}})(\sqrt{V(0)}-\frac{\lambda_{max}\Vert P\Vert  k}{\epsilon \sqrt{\lambda_{min}}}).
\eaqn
Thus, one obtains $(\ref{estia})$. On the other hand, from (\ref{bounded}) we have 
$$
\frac{dV}{dt}\le-\epsilon \Vert e\Vert ^2  +k \Vert P\Vert  \Vert e\Vert.
$$
If $\Vert e\Vert\ge \frac{k \Vert P\Vert }{\epsilon}+\rho$ for some $\rho>0$, one has $\frac{dV}{dt}\le-\rho \epsilon \Vert e\Vert $. Hence classically $\Vert e\Vert$ converges to any neighborhood of $\Omega$ in finite time and stay there.  \\

\noindent (b) In particular, $\kappa_2(\cdot)$ is bounded by some $k>0$. Note that from $(\ref{estia})$, we imply that
\beq\label{estia1}
\Vert e(t) \Vert \le\frac{\lambda_{max}\Vert P\Vert  k}{ {\lambda_{min}}\epsilon}+{\rm exp}(\frac{-\epsilon t}{\lambda_{max}})\sqrt{\frac{V(0)}{\lambda_{min}}},\;\;\forall t\ge 0.
\eeq
In particular, $e(t)$ is bounded and hence   $V(e(t))$ is bounded by some $V^*$. {Let $T>0$  be given, $T_n=nT$ and 
$$c_n:=\sup_{t\in [T_n, +\infty)}\vert \kappa_2(t)\vert, \; n\ge 1.$$
 Then $(c_n)$ is bounded by $k$ and converges to zero since $\kappa_2(t)\to 0$ as $t\to \infty$.} Using (\ref{estia1}), by replacing $0$ by $T_n$, for all $t\ge T_n$, we obtain
\beq
\Vert e(t) \Vert \le\frac{\lambda_{max}\Vert P\Vert  c_n}{ {\lambda_{min}}\epsilon}+{\rm exp}(\frac{-\epsilon (t-T_n)}{\lambda_{max}})\sqrt{\frac{V^*}{\lambda_{min}}}.
\eeq
For given $\delta>0$, we can choose $n_0$ such that $\frac{\lambda_{max}\Vert P\Vert  c_{n_0}}{ {\lambda_{min}}\epsilon}\le \delta$. Then 
\beq
\Vert e(t) \Vert \le\delta+{\rm exp}(\frac{-\epsilon (t-T_{n_0})}{\lambda_{max}})\sqrt{\frac{V^*}{\lambda_{min}}}, \;\;\forall \;t\ge T_{n_0}.
\eeq
Thus
$$
\limsup_{t\to\infty}\Vert e(t) \Vert \le\delta.
$$
Since $\delta$ is arbitrary, we must have 
$$
\limsup_{t\to\infty}\Vert e(t) \Vert \le0,
$$
and thus
$$
\lim_{t\to\infty}\Vert e(t) \Vert =0.
$$
\noindent (c) If $\Vert \kappa_2(t)\Vert\le k e^{-at}$ {for  all $t\ge 0$}, then
\baq
\sqrt{V(t)}&\le& \sqrt{V(0)}{\rm exp}(\frac{-\epsilon t}{\lambda_{max}})+\frac{\Vert P\Vert  k}{ \sqrt{\lambda_{min}}}\int_{0}^t{\rm exp}(\frac{\epsilon (s-t)}{\lambda_{max}}-as)ds\\
&=&\sqrt{V(0)}{\rm exp}(\frac{-\epsilon t}{\lambda_{max}})+\frac{\lambda_{max}\Vert P\Vert  k}{ \sqrt{\lambda_{min}}(\epsilon-a\lambda_{max})}(e^{-at}-{\rm exp}(\frac{-\epsilon t}{\lambda_{max}})),
\eaq
if $\epsilon\neq a\lambda_{max}$.
Thus the error $e(t)$ converges to zero with exponential rate. The case $\epsilon= a\lambda_{max}$ is trivial. \\

\noindent (d) If $k_2\in L^2(0,\infty )$, then from (\ref{totales}), one has 
\baq\nonumber
\sqrt{V(t)}&\le& \sqrt{V(0)}{\rm exp}(\frac{-\epsilon t}{\lambda_{max}})+\frac{\Vert P\Vert }{\sqrt{\lambda_{min}}}\sqrt{\int_0^t \kappa^2_2(s)ds}\sqrt{\int_{0}^t{\rm exp}(\frac{2\epsilon (s-t)}{\lambda_{max}})ds}\\\nonumber
&=& \sqrt{V(0)}{\rm exp}(\frac{-\epsilon t}{\lambda_{max}})+\frac{\Vert P\Vert  }{\sqrt{\lambda_{min}}}\sqrt{\int_0^t \kappa^2_2(s)ds}\sqrt{\frac{\lambda_{max}}{2\epsilon}(1-{\rm exp}(\frac{-2\epsilon t}{\lambda_{max}}))}\\
&\le &\sqrt{V(0)}{\rm exp}(\frac{-\epsilon t}{\lambda_{max}})+\frac{\Vert P\Vert \Vert \kappa_2\Vert_{L^2(0,\infty)}}{\sqrt{\lambda_{min}}}\sqrt{\frac{\lambda_{max}}{2\epsilon}(1-{\rm exp}(\frac{-2\epsilon t}{\lambda_{max}}))}.
\eaq
{The proof is thereby completed.}
\end{proof}
 \begin{remark}\normalfont
 i) In the first case, we have the interval estimation for the error, i.e., the original state $x(t)$ belongs to some known interval centered at the approximate observer $\tilde{x}(t)$.  In the remaining cases, one obtains  the exact observer. It covers the case where the unknown $\xi$ vanishes after some time instant $T_0>0$ as considered in  some examples in \cite{Huang3}. Note that with our observer, we only require that the unknown part $\xi_2$ in the complement of the observation space tends to vanish (see Figure 3, Example 2).  \\
\noindent ii) The term $\frac{\kappa_3 e_y}{\Vert e_y\Vert^2+\delta }$ makes $e_y$ small  very fast  and remains small.
Note that if $\Vert e_y\Vert^2\ge \delta$, then $\frac{ e^2_y}{\Vert e_y\Vert^2+\delta}\ge \frac{1}{2}$. Thus, if we choose $\kappa_3(t) >\frac{(\Vert P\Vert \Vert \kappa_2(t)\Vert +\rho)^2}{2\epsilon} $ for some $\rho>0$ and $\Vert e_y\Vert^2\ge \delta$,  then 
\baq\nonumber
\frac{1}{2}\frac{dV}{dt}&=&\langle P\dot{e}, e\rangle=\langle P(A-LF)e+PB(\tilde{\omega}-{\omega}), e\rangle+\langle P(f_1(\tilde{x},u)-f_1({x},u)), e \rangle \\\nonumber
& -&(\kappa_1(t)-\Vert\xi_1(t)\Vert){\Vert e_y\Vert} -\frac{ \kappa_3(t) e^2_y}{\Vert e_y\Vert^2+\delta}- \langle P\xi_2(t),e\rangle\\
&\le& -\epsilon \Vert e \Vert^2 - \frac{\kappa_3(t)}{2}+ \Vert P\Vert \Vert\kappa_2(t)\Vert \Vert e \Vert \le -\rho \Vert e \Vert.
\eaq
Then $\Vert e(t) \Vert$ decreases very fast in finite time such that $\Vert e_y(t) \Vert^2<\delta$.\\
iii) Suppose that $Pk_2(\cdot)$ is bounded by $k'$, then we can improve the  set $\Omega$ by the new attractive set $\Omega'=[0,\frac{k'}{\epsilon}]$.\\
iv) The same result is obtained if the unknown $\xi(t)$ is replaced by $\xi(t,x,u)$.
\end{remark}
Let us recall the definition of  $H^\infty$ observer (see, e.g., \cite{Huang3}) and {introduce the notion of strong $H^\infty$ observer as well as $T$-observer.}
\begin{definition}
 If the error $t\mapsto e(t)$  is asymptotically stable for $\xi\equiv 0$ and $\Vert e \Vert_{L^2(0,\infty)}\le  \mu \Vert \xi \Vert_{L^2(0,\infty)}$ for some $\mu>0$ under zero initial condition (i.e., $e(0)=0$) for non-zero $\xi \in L^2(0,\infty)$, then (\ref{obs}) is said {an $H^\infty$} observer for the system  (\ref{sysh}).
\end{definition}

{\begin{definition}
The observer (\ref{obs}) is called {a } strong $H^\infty$ observer for the system  (\ref{sysh}) if the error $t\mapsto e(t)$  is asymptotically stable for $\xi\equiv 0$ and for non-zero $\xi \in L^2(0,\infty)$, there exists $\mu>0$ such that  
\beq
\Vert e \Vert_{L^2(0,\infty)}\le  C+\mu \Vert \xi \Vert_{L^2(0,\infty)},
\eeq
where $C$ is a non-negative constant that vanishes under zero initial condition.
\end{definition}}
{\begin{definition}
The observer (\ref{obs}) is called {a } $T$- observer for the system  (\ref{sysh}) if the error $t\mapsto e(t)$  is asymptotically stable for $\xi\equiv 0$ and for non-zero $\xi \in L^2(0,\infty)$, there exists $\mu>0$ such that  
\beq
\Vert e(t) \Vert \le \gamma(t)+\mu\sqrt{\int_0^t \xi^2(s)ds}, \;\;\forall \;t\ge 0,
\eeq
where $\gamma(t)$ vanishes under zero initial condition and converges exponentially to zero for the remaining case.
\end{definition}
\begin{remark}\normalfont
It is easy to see that a strong $H^\infty$ observer is also an  $H^\infty$ observer while $T$- observer provides a time instant estimation for the error. 
\end{remark}}
\noindent \textbf{Assumption 5:} There exists some $\mu>0$  such that 
\baq\label{block}
\left( \begin{array}{ccc}
\Omega &\;\; -P \\ \\
-P &\;\; -2\mu \epsilon I
\end{array} \right)\le 0
\eaq
where $\Omega:=P(A-LF)+(A-LF)^TP+2L_f \Vert P \Vert I+2 \epsilon I$ and $\epsilon>0$ is in Assumption 4.
{\begin{theorem}
Suppose that Assumptions 1--4 hold. Then (\ref{obs}) is  a $T$ - observer for the system  (\ref{sysh}). In addition, if Assumption 5 holds, then (\ref{obs}) is  {a strong $H^\infty$} observer for the system  (\ref{sysh}).
\end{theorem}}
\begin{proof}
From Theorem \ref{mainth}, if the uncertainty $\xi \equiv 0$, then the error $t\mapsto e(t)$ converges to zero exponentially. For non-zero $\xi \in L^2(0,\infty)$, using Theorem \ref{mainth}-d, we have 
\baq\nonumber
\lambda_{min} \Vert e(t) \Vert &\le& \sqrt{V(t)}\le  \sqrt{V(0)}{\rm exp}(\frac{-\epsilon t}{\lambda_{max}})+\frac{\Vert P\Vert  }{\sqrt{\lambda_{min}}}\sqrt{\int_0^t \xi^2_2(s)ds}\sqrt{\frac{\lambda_{max}}{2\epsilon}(1-{\rm exp}(\frac{-2\epsilon t}{\lambda_{max}}))}\\
&\le &\sqrt{V(0)}{\rm exp}(\frac{-\epsilon t}{\lambda_{max}})+\frac{\Vert P\Vert  }{\sqrt{\lambda_{min}}}\sqrt{\frac{\lambda_{max}}{2\epsilon}}\sqrt{\int_0^t \xi^2(s)ds},
\eaq
which deduces that (\ref{obs}) is  a $T$ - observer for the system  (\ref{sysh}). If Assumption 5 is satisfied, from (\ref{dv})-(\ref{Lips}), with the same $V(t)=\langle Pe(t),e(t)\rangle$, we have
\baq\nonumber
\frac{1}{2}\frac{dV}{dt}&=&\langle P\dot{e}, e\rangle\\\nonumber
&=&\langle P(A-LF)e+PB(\tilde{\omega}-{\omega}), e\rangle\\\nonumber
&+&\langle P(f_1(\tilde{x},u)-f_1({x},u)), e \rangle -\kappa_1(t){\Vert e_y\Vert} -\frac{\kappa_3 \Vert e_y \Vert^2}{\Vert e_y\Vert^2+\delta}- \langle P\xi,e\rangle\\\nonumber
&\le&\langle P(A-LF)e, e\rangle+L_f \Vert P \Vert \Vert e \Vert^2- \langle P\xi,e\rangle\\
&\le& \epsilon ( -\Vert e \Vert^2+ \mu \Vert \xi \Vert^2).
\eaq
  Integrating both side from $0$ to $\infty$ and note that $V$ is non-negative, we obtain that $\Vert e \Vert_{L^2(0,\infty)}\le \frac{V(0)}{2 \epsilon}+ \mu \Vert \xi \Vert_{L^2(0,\infty)}$. It means that (\ref{obs}) is  {a strong $H^\infty$} observer for the system  (\ref{sysh}).
\end{proof}
\begin{remark}\normalfont
i) Assumption 5 implies that $\Omega=P(A-LF)+(A-LF)^TP+2L_f \Vert P \Vert I+2 \epsilon I\le 0$, which is consistent to Assumption 4. In $\Omega$, we use the term $\epsilon I$ instead of $I$ as in \cite{Huang3}, which is easier to have the non-positiveness of $\Omega$.\\
ii) The positive number $\mu$  exists, for example, if $\Omega \le -2\epsilon I $, we can choose $\mu > \frac{\Vert P \Vert^2}{4\epsilon^2}$.
\end{remark}
\section{Numerical examples}
 \begin{example}\normalfont
  First we consider the system (\ref{sysh}) with 
  $$
  A=\left( \begin{array}{ccc}
-6 &\;\; 4&\;\; 0 \\ \\
7 &\;\; -8 &\;\;  0 \\ \\
0&\;\;0 &\;\;-7
\end{array} \right), B=\left( \begin{array}{ccc}
4  \\ \\
6  \\ \\
-3
\end{array} \right), f(x,u)=\left( \begin{array}{ccc}
u+ 2\sin x_2  \\ \\
2u+ 3 \cos x_1 \\ \\
-u+4\sin x_3
\end{array} \right)
  $$
  
  $$
   C=\left( \begin{array}{ccc}
8 &\;\; 6&\;\; -3
\end{array} \right), F = \left( \begin{array}{ccc}
1 &\;\; 0&\;\; 0
\end{array} \right), \;\; u=5\sin t.
  $$
 Suppose that the unknown  $\xi(t,x,u)=\left( \begin{array}{ccc}
2 \\ \\
5\cos x_1 \\ \\
4 \sin t
\end{array} \right)$ and 
 $$
 \mathcal{F}(x) = \left\{
\begin{array}{lll}
{\rm sign}(x)(3\vert x \vert +6) & \mbox{ if } & x \neq 0,\\
&&\\
\; [ -6,6 ] & \mbox{ if } & x = 0.
\end{array}\right.
$$
Then $L=4$ and $k=\sqrt{41}$. Then Assumptions 1--4 are satisfied with 
  $$
 P=\left( \begin{array}{ccc}
1 &\;\; 0&\;\; 0 \\ \\
0 &\;\; 1 &\;\;  0 \\ \\
0&\;\;0 &\;\;1
\end{array} \right),\;\;\; L=\left( \begin{array}{ccc}
0  \\ \\
11 \\ \\
0
\end{array} \right), \;\;\;\epsilon = 2, \;\;\; K=4.
  $$
One can {observe} that the total error $\Vert e\Vert$ converges  to any neighborhood of the set $\Omega =[0,\sqrt{41}/2]$ in finite time  while the observed error $e_1$ tends to zero very fast. Consequently, the observer state $\tilde{x}$  using  (\ref{obs}) provides a good estimation for the original state $x$. 
   \begin{figure}[h!]
\begin{center}
\includegraphics[scale=0.56]{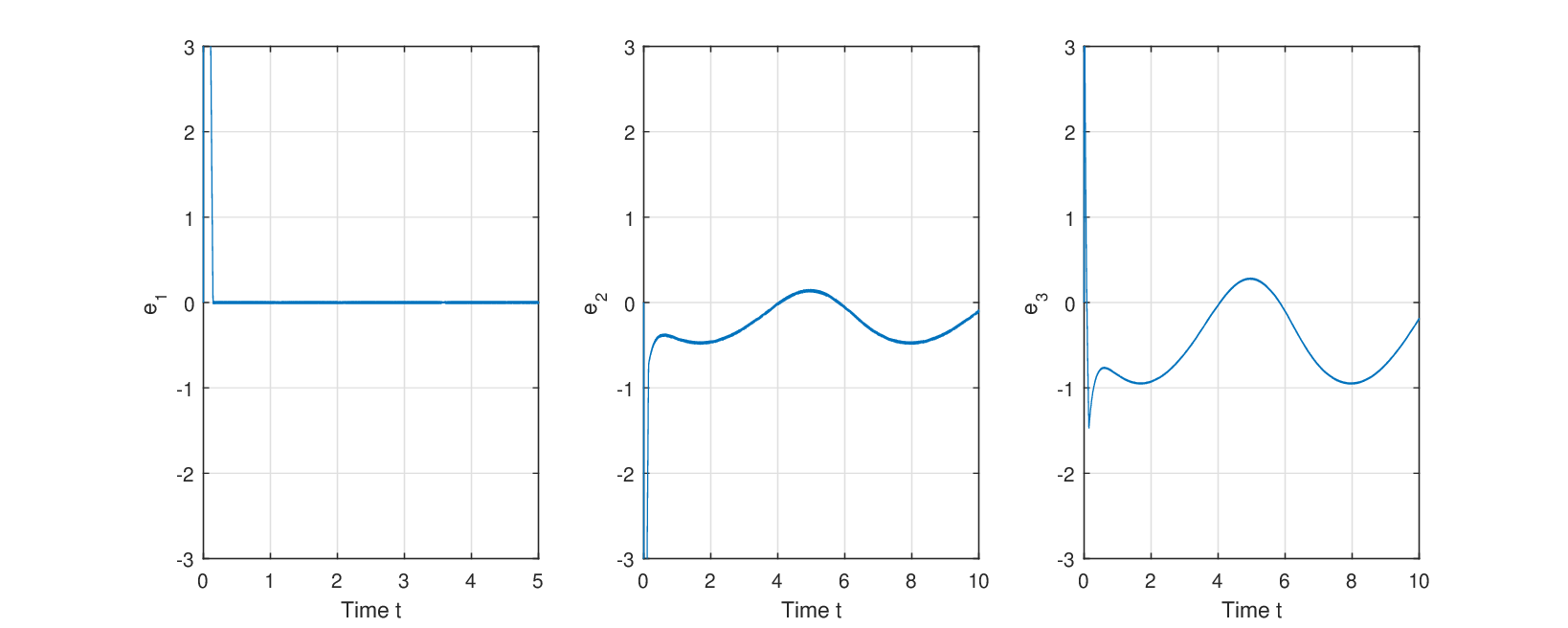}
\caption{Errors using  the approximate sliding mode observer (\ref{obs})} 
\end{center}
\end{figure}
 \end{example}
 
  \begin{example}\normalfont
  Next let us consider the  rotor system with friction as in \cite{Bruin,Juloski}.
Let $x_{1}=\theta_{u}-\theta_{l}$, $x_{2}=\dot{\theta}_{u}$ and $x_{3}=\dot{\theta}_{l}$ where $\theta_{u}$ and $\theta_{l}$ are the angular positions of the upper and lower discs, respectively, we have
\begin{equation}\label{equ:H2}
\begin{split}
\dot{x}_{1}&=x_{2}-x_{3}\\
\dot{x}_{2}&=\frac{k_{m}}{J_{u}}u-\frac{k_{\theta}}{J_{u}}x_{1}-\frac{1}{J_{u}}T_{fu}(x_{2})\\
\dot{x}_{3}&=\frac{k_{\theta}}{J_{l}}x_{1}-\frac{1}{J_{l}}T_{fl}(x_{3}),
\end{split}
\end{equation}
where $u$ is the input voltage to the power amplifier of the motor, 
$T_{fu}(x_{2})=b_{up}x_{2}$
and
\begin{equation}\label{equ:H4}
T_{fl}(x_{3})=\left\{\begin{array}{lr}
$[$T_{sl}+T_{1}(1-\frac{2}{1+e^{\omega_{1}|x_{3}|}})\\
+T_{2}(1-\frac{2}{1+e^{\omega_{1}|x_{3}|}})$]${\rm sign}(x_{3})+b_{l}x_{3}, &x_{3}\neq0\\
$[$-T_{sl},~T_{sl}$]$, &x_{3}=0.
\end{array}\right.
\end{equation}
Then \eqref{equ:H2} can be written as
\begin{equation}\label{equ:H5}
\begin{split}
\dot{x}_{1}&=x_{2}-x_{3}\\
\dot{x}_{2}&=\frac{k_{m}}{J_{u}}u-\frac{k_{\theta}}{J_{u}}x_{1}-\frac{b_{up}}{J_{u}}x_{2}\\
\dot{x}_{3}&=\frac{k_{\theta}}{J_{l}}x_{1}-\frac{1}{J_{l}}T_{fl}(x_{3}).
\end{split}
\end{equation}
The estimation of the parameters are given in  \cite[Table 1]{Juloski}.
 Substituting the values into the functions, we have
\begin{equation*}
T_{fu}(x_2)=2.2247x_2,
\end{equation*}
\begin{equation*}
T_{fl}(x_{3})=\left\{\begin{array}{lr}
$[$0.1642+0.0603(1-\frac{2}{1+e^{5.7468|x_{3}|}})\\
-0.2267(1-\frac{2}{1+e^{0.2941|x_{3}|}})$]${\rm sign}(x_{3})+0.0109x_{3}, &x_{3}\neq0\\
$[$-0.1642,~0.1642$]$, &\;\;x_{3}=0.
\end{array}\right.
\end{equation*}
We can rewrite   (\ref{equ:H5}) as follows 
\begin{equation}\label{equ:H6}
\left\{\begin{array}{lr}
\dot{x}=Ax+B\omega+Gu\\
\omega\in -T_{fl}(Cx)\\
y=Fx,
\end{array}\right.
\end{equation}
where $A=\left[\begin{array}{ccc}
0&1&-1\\
-\frac{k_{\theta}}{J_{u}}&-\frac{b_{up}}{J_{u}}&0\\
\frac{k_{\theta}}{J_{l}}&0&0\\
\end{array}\right]$, $G=\left[\begin{array}{c}
0\\
\frac{k_{m}}{J_{u}}\\
0
\end{array}\right]$, $B=\left[\begin{array}{c}
0\\
0\\
\frac{1}{J_{l}}
\end{array}\right]$, $C=[0 ~0 ~1]$, $F=[1~ 0 ~0]$.
Substituting the estimated values of the parameters, we have
\begin{equation*}
A=\left[\begin{array}{ccc}
0&1&-1\\
-0.1526&-4.6688&0\\
2.2301&0&0\\
\end{array}\right],\quad G=\left[\begin{array}{c}
0\\
8.3841\\
0
\end{array}\right],\quad B=\left[\begin{array}{c}
0\\
0\\
30.6748
\end{array}\right].
\end{equation*}
Since $T_{fl}(x_{3})$ is not monotone, we use the loop transformation $\tilde{T}_{fl}(Cx)=T_{fl}(Cx)-mCx$, where $m=-0.021$. Then
\begin{equation*}
\tilde{T}_{fl}(x_{3})=\left\{\begin{array}{lr}
$[$0.1642+0.0603(1-\frac{2}{1+e^{5.7468|x_{3}|}})\\
-0.2267(1-\frac{2}{1+e^{0.2941|x_{3}|}})$]${\rm sign}(x_{3})+0.0319x_{3}, &x_{3}\neq0\\
$[$-0.1642,~0.1642$]$, &x_{3}=0.
\end{array}\right.
\end{equation*}
Under the transformation, we have a new system
\begin{equation}\label{equ:H6}
\left\{\begin{array}{lr}
\dot{x}=\bar{A}x+B\omega+Gu\\
\omega\in -\mathcal{F}(Cx)\\
y=Fx,
\end{array}\right.
\end{equation}
%
where
\begin{equation*}
\bar{A}=A-mBC=\left[\begin{array}{ccc}
0&1&-1\\
-0.1526&-4.6688&0\\
2.2301&0&0.6442\\
\end{array}\right],\quad G=\left[\begin{array}{c}
0\\
8.3841\\
0
\end{array}\right],
\end{equation*}
\begin{equation*}
\quad B=\left[\begin{array}{c}
0\\
0\\
30.6748
\end{array}\right],  \quad C=[0 ~0 ~1],\quad F=[1~ 0 ~0].
\end{equation*}
\begin{equation*}
\mathcal{F}(\lambda)=\left\{\begin{array}{lr}
$[$0.1642+0.0603(1-\frac{2}{1+e^{5.7468|\lambda|}})\\
~~~~-0.2267(1-\frac{2}{1+e^{0.2941|\lambda|}})$]${\rm sign}(\lambda)+0.0319\lambda, &\lambda\neq0\\
$[$-0.1642,~0.1642$]$, &\lambda=0.
\end{array}\right.
\end{equation*}
Here we choose $u=2$.
Solving the following LMIs:
\begin{equation}\label{equ:H7}
\begin{split}
(\bar{A}-LF)^{T}P+P(\bar{A}-LF)&=-Q\\
B^TP=C-KF,
\end{split}
\end{equation}
we  obtain
\begin{equation*}
P=\left[\begin{array}{ccc}
0.2958&0.0417&0.0600\\
0.0417&0.0286&0\\
0.0600&0&0.0326\\
\end{array}\right],~Q=\left[\begin{array}{ccc}
-0.1327&-0.0064&-0.0158\\
-0.0064&-0.1837&0.0183\\
0.0158&0.0183&-0.0780
\end{array}\right],
\end{equation*}
\begin{equation*}
L=\left[\begin{array}{c}
3.3069\\
-1.2140\\
-12.2290
\end{array}\right],~K=-1.8392, ~\epsilon=0.0714.
\end{equation*}
{Now  suppose that the system (\ref{equ:H6}) is influenced by an uncertainty $\xi(t,x)$, which is described as follows:}
\begin{equation}\label{equ:H8}
\left\{\begin{array}{lr}
\dot{x}=\bar{A}x+B\omega+Gu+\xi(t,x)\\
\omega\in -\mathcal{F}(Cx)\\
y=Fx.
\end{array}\right.
\end{equation}
First we take the unknown  $\xi(t,x)=\xi^{(1)}=\left( \begin{array}{ccc}
1 \\ \\
4\sin x_2 \\ \\
 \cos t
\end{array} \right)$.
   \begin{figure}[h!]
\begin{center}
\includegraphics[scale=0.5]{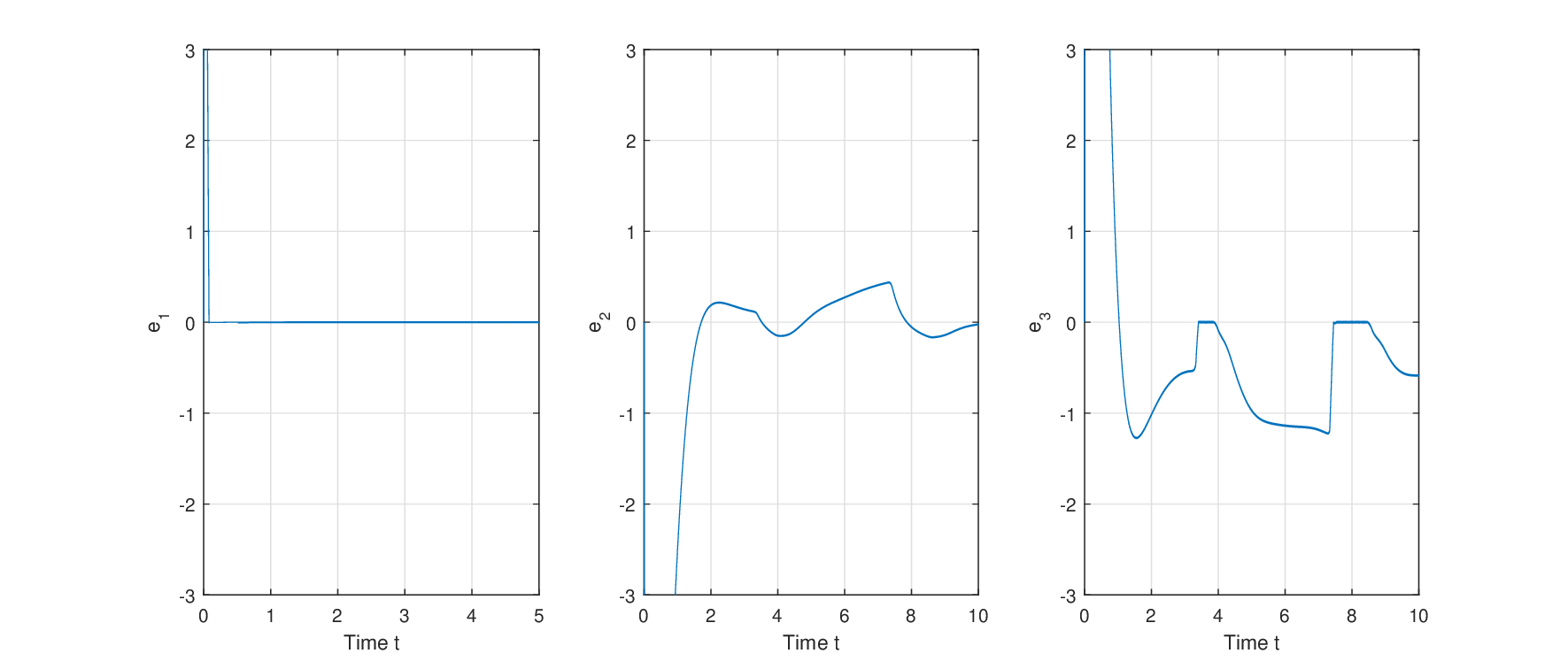}
\caption{Errors using  (\ref{obs}) applied to Example 2 with the uncertainty $\xi^{(1)}$} 
\end{center}
\end{figure}
Using the approximate sliding mode observer (\ref{obs}), we can see that observation error $e_1$ converges to zero very fast and the total error module $\Vert e(t) \Vert$ belongs to the attractive set $\Omega'=[0,7.75]$ in finite time (Figure 2). The numerical result is consistent with Theorem \ref{mainth}  when the uncertainty is chosen randomly. 
Next we consider  the unknown  $$\xi(t,x)=\xi^{(2)}=\left( \begin{array}{ccc}
16.0552 \\ \\
-23.4092 \\ \\
-29.5495
\end{array} \right)+\left( \begin{array}{ccc}
e^{-t} \\ \\
e^{-2t}\\ \\
e^{-1.5t}
\end{array} \right),$$ which has 2 parts: the first part in ${\rm Im}(P^{-1}F^T)$ and the second part vanishes when the time is large. Then the  sliding mode observer (\ref{obs}) is indeed the observer of the original system (Figure 3).
 \begin{figure}[h!]
\begin{center}
\includegraphics[scale=0.5]{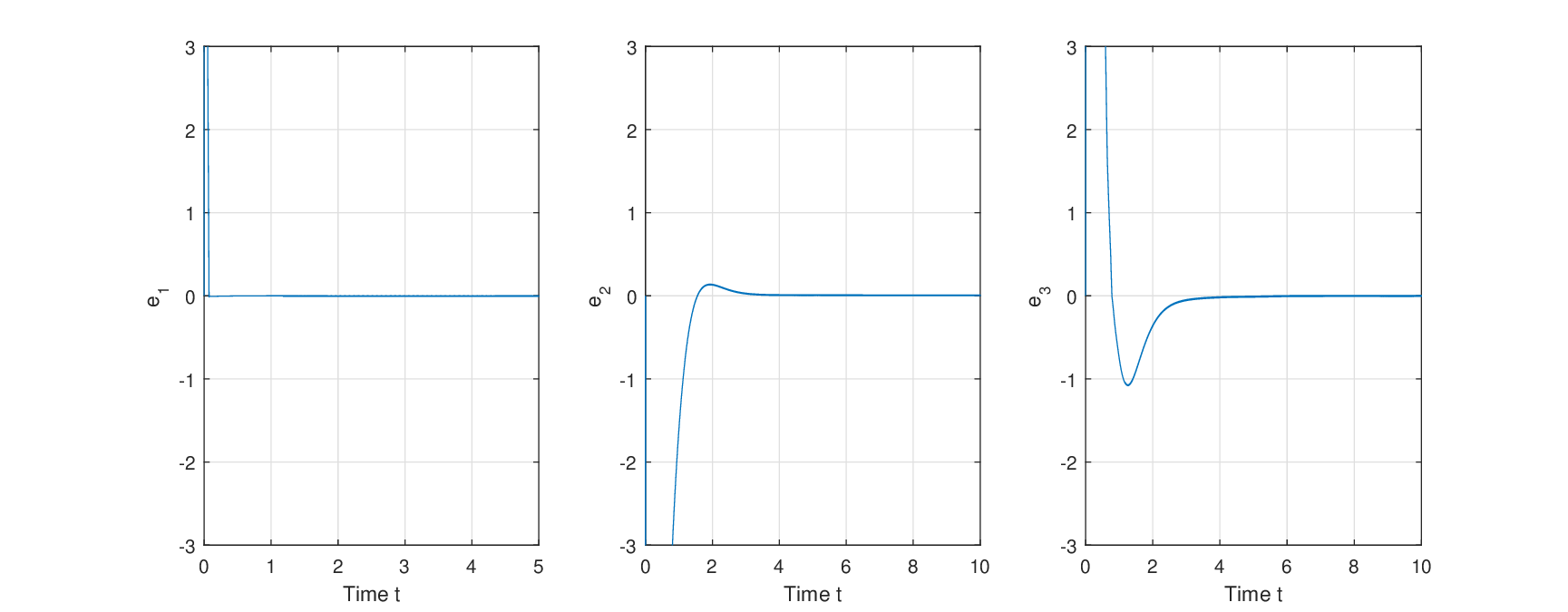}
\caption{Errors using  (\ref{obs}) applied to Example 2 with the uncertainty $\xi^{(2)}$} 
\end{center}
\end{figure}
 \end{example}
 \section{Conclusions}
{
In this paper, we introduce a {sliding mode observer, which also functions as a $T$- observer and a strong $H^\infty$ observer, designed for a general category of set-valued Lur'e dynamical systems}. Importantly, our approach accommodates scenarios where the uncertainty does not fall within the range of observation. Consequently, we are able to obtain accurate state estimations for the original systems. Furthermore, in cases where the unobservable portion of uncertainty tends to vanish when the time is large, we achieve exact observers. 
It would be interesting to investigate strategies that enhance the robustness of our proposed observer, enabling it to handle increasingly complex and dynamic uncertainties or disturbances. Exploring ways to incorporate adaptive control strategies into our observer framework would be a promising direction, facilitating its adaptation to evolving system conditions. Optimizing and reducing the size of the attractive set could be a valuable pursuit. The practical application and validation of our observer in real-world systems, such as autonomous vehicles, robotics, or industrial processes, present good opportunities for future exploration. 
Extending our observer framework to address the challenges posed by  set-valued uncertainties would further broaden its applicability and impact. 
These open questions require further investigation and are beyond the scope of the current paper. They will be the focus of a new research project.
}

\end{document}